\documentclass{amsart}
\usepackage{graphicx,amsfonts,amssymb,amsmath,amsthm}
\usepackage{cite}
\usepackage[small,nohug,heads=vee]{diagrams}
\diagramstyle[labelstyle=\scriptstyle]
\usepackage{pdfsync}

\theoremstyle{plain}
\newtheorem{theorem}    {Theorem}
\newtheorem{lemma}      [theorem]{Lemma}

\newtheorem{proposition}[theorem]{Proposition}

\theoremstyle{definition}

\theoremstyle{remark}
\newtheorem{remark}              {Remark}

\numberwithin{equation}{section}

\def\A{\mathbb A}
\def\C{\mathbb C}
\def\F{\mathbb F}

\def\Q{\mathbb Q}

\def\Z{\mathbb Z}

\begin{document}

\title{Further refinement of strong multiplicity one for GL(2)}

\author{Nahid Walji}
\address{\'{E}cole Polytechnique F\'{e}d\'{e}rale de Lausanne, Station 8, CH-1015 Lausanne, Switzerland}
\email{nahid.walji@epfl.ch}
\maketitle
\begin{abstract}
We obtain a sharp refinement of the strong multiplicity one theorem for the case of unitary non-dihedral cuspidal automorphic representations for GL(2).
Given two unitary cuspidal automorphic representations for GL(2) that are not twist-equivalent, we also find sharp lower bounds for the number of places where the Hecke eigenvalues are not equal, for both the general and non-dihedral cases.
We then construct examples to demonstrate that these results are sharp.
\end{abstract}

\section*{Introduction}
\label{intro}
Given a number field $F$, let $\mathcal{A}_0 ({\rm GL}_n(\A_F)) $ be the set of cuspidal automorphic representations for $ {\rm GL}_n(\A_F)  $ with unitary central character. For $\pi \in \mathcal{A}_0 ({\rm GL}_n(\A_F))$, at any finite place $v$ of $F$ where $\pi$ is unramified we denote the Langlands conjugacy class by $A_v (\pi)  \subset  {\rm GL}_n(\C) $, which we represent by the matrix $ {\rm diag}\{\alpha _{1, v }, \dots, \alpha _{n,v} \}$. Let $a _v (\pi) $ denote the trace of this matrix.

Let us now set $n = 2$. Given $\pi,\pi' \in \mathcal{A}_0 ({\rm GL}_2(\A_F)) $, one can compare local data to determine whether $\pi$ and $\pi'$ are globally isomorphic. In this context we ask the following question: if we have a set $S$ such that $a _v (\pi) = a _v (\pi') $ for all $v \not \in  S$, what property of $S$ is sufficient to establish that $\pi$ and $\pi'$ are globally isomorphic?

One approach involves establishing a condition on the size of $S$. The strong multiplicity one theorem of Jacquet--Shalika~\cite{JS81} states that it is sufficient for $S$ to be finite. In 1994, D.~Ramakrishnan~\cite{Ra94} proved that a set $S$ of density less than 1/8 is sufficient. This bound is sharp, as shown by an example of J.-P. Serre~\cite{Se77} of a pair of dihedral automorphic representations (see Section~\ref{s4-4} for details).

This naturally leads to the question of whether the bound can be increased if we exclude dihedral automorphic representations. 
Given $\pi, \pi' \in \mathcal{A}_0({\rm GL}_2(\A_F))$, let us define the set 
\begin{align*}
S = S (\pi,\pi') := \{ v \mid v \text{ unramified for }\pi \text{ and }\pi', a_v(\pi) \neq a_v(\pi')  \}.
\end{align*}
We will show:

\begin{theorem} \label{t1}
Let $\pi, \pi ' \in  \mathcal{A}_0 ({\rm GL}_2(\A_F))$ be distinct non-dihedral 
 representations.
 Then, if $\underline{\delta} (S)$ is the lower Dirichlet density of the set $S = S (\pi, \pi')$, we have
  \begin{align*}
\underline{\delta} (S) \geq \frac 14.
  \end{align*}
\end{theorem}

\begin{remark}
This bound is sharp, as established in Section~\ref{s4-1} by a pair of octahedral automorphic representations.

The question of what bound can be established when exactly one of the representations is dihedral is addressed later by Theorem~\ref{t3}.
\end{remark}

A related question involves weakening the requirement of global isomorphism to that of twist-equivalence. Specifically, we can ask: given a set $S$ such that $a _v (\pi) = a _v (\pi') $ for all $v \not \in  S $, what property of $S$ is sufficient to establish that $\pi$ and $\pi'$ are twist-equivalent?

We will prove:
\begin{theorem} \label{t2}
Let $\pi, \pi ' \in  \mathcal{A}_0 ({\rm GL}_2(\A_F))$ be representations that are not twist-equivalent.
For $S = S (\pi,\pi')$, we have
  \begin{align*}
\underline{\delta} (S) \geq \frac 29.
  \end{align*}
\end{theorem}

\begin{remark}
This bound is sharp, as demonstrated by a pair of dihedral automorphic representations associated to the group $S_3$ (see Section~\ref{s4-2}).
\end{remark}

We ask if the bound can be increased if we specify that one, or both, of the cuspidal automorphic representations is non-dihedral. We obtain
\begin{theorem} \label{t3}
Let $\pi, \pi ' \in  \mathcal{A}_0 ({\rm GL}_2(\A_F))$ be 
non-twist-equivalent. For $S = S (\pi, \pi')$:\\
(i) If exactly one of the representations is non-dihedral, then
  \begin{align*}
\underline{\delta} (S) \geq \frac 27.
  \end{align*}
(ii) If both of the representations are non-dihedral, then
  \begin{align*}
\underline{\delta} (S) \geq \frac 25.
  \end{align*}
\end{theorem}

\begin{remark}
The second bound is sharp, as demonstrated by a pair of odd icosahedral automorphic representations (see Section~\ref{s4-3}).

Note that the first bound applies to all pairs consisting of a dihedral automorphic representation and a non-dihedral cuspidal automorphic representation, since the representations will always be twist-inequivalent.
It does not appear that this bound is sharp. This may be because, unlike in the other three cases, we have forcibly broken the symmetry in the pair - one is dihedral and the other is not. 

\end{remark}

Given that Theorem~\ref{t1} is sharp, one might ask what analogous theorem to expect if we only consider cuspidal automorphic representations that are neither dihedral nor octahedral. Following the same method as in the proof of Theorem~\ref{t1} does not seem to lead to improved bounds in this case. However, based on examples, we expect that there would be a lower bound of 3/8 for $\underline{\delta}(S)$, which would be sharp in both the tetrahedral and icosahedral cases.

Beyond that, if we restrict ourselves to non-polyhedral cuspidal automorphic representations, it is natural to conjecture 
that two such representations are globally isomorphic if they agree locally for a set of places of density greater than 1/2. This would be sharp due to the following example: Let $\pi$ be a cuspidal automorphic representation that corresponds to a non-dihedral holomorphic newform of weight greater than 1. We observe that the condition on the weight, along with being non-dihedral, then implies that the newform is of non-polyhedral type. One knows that the set of primes where the Hecke eigenvalue of $\pi$ is zero has density zero~\cite{Se81}. Thus $\pi$ and $\pi \otimes \chi$, where $\chi$ is a quadratic Hecke character, provide an example of a pair of non-polyhedral cuspidal automorphic representations that agree locally for a set of places of density exactly 1/2. Note that the holomorphy condition is needed here: if $\pi$ were associated to a non-dihedral Maa{\ss} form then the best known upper bounds on the density of places at which the Hecke eigenvalue is zero is 1/2~\cite{Wa13b}.

In the case of Theorem~\ref{t3}, restricting further to cuspidal automorphic representations that are neither dihedral nor icosahedral does not seem to yield better bounds under the current approach. However, examples indicate that it may be reasonable to expect the existence of a lower bound of 15/32 for $\underline{\delta}(S)$, which would be optimal in the tetrahedral case.
For two non-polyhedral cuspidal automorphic representations that are not twist-equivalent, it is natural to conjecture that they agree locally for a set of places of density 0.

We collect all the examples implicit in the discussions above:
\begin{theorem}\label{obs}
We note the existence of the following cuspidal automorphic representations for ${\rm GL}(2)$:
\begin{itemize}
\item (due to Serre~\cite{Se77}) A pair of dihedral representations with $\delta (S) = 1/8$. 
\item A pair of dihedral representations that are non-twist-equivalent and where $\delta (S) = 2/9$.
\item A pair of tetrahedral representations with $\delta (S) = 3/8$.
\item A pair of tetrahedral representations that are non-twist-equivalent and where $\delta (S) = 15/32$.
\item A pair of octahedral representations with $\delta (S) = 1/4$.
\item A pair of icosahedral representations with $\delta (S) = 3/8$.
\item A pair of icosahedral representations that are non-twist-equivalent and where $\delta (S) = 2/5$.
\item A pair of non-polyhedral representations with $\delta (S) = 1/2$.
\end{itemize}
\end{theorem}
These statements will be proved in Section~\ref{s4}.\\

The idea of the proof of Theorems~\ref{t1},~\ref{t2}, and~\ref{t3} relies on the examination of the asymptotic properties of certain suitable Dirichlet series.
 We achieve this by establishing various identities of incomplete $L$-functions and determining their behaviour when real $s \rightarrow 1^+$ ($s$ being the variable in each incomplete $L$-function) by using functoriality and non-vanishing results.
We make use of the work of Gelbart--Jacquet~\cite{GJ78} and Kim--Shahidi~\cite{KS00} on the symmetric second and third powers (respectively) of cuspidal automorphic representations for GL(2) with regard to their existence and possible cuspidality (note however that the symmetric fourth power lift of Kim~\cite{Ki03} is not needed).
We also require the results of Ramakrishnan~\cite{Ra00,Ra04} on the adjoint lift and the automorphic tensor product.
Our approach is related to, though is not a straightforward extension of, work of Ramakrishnan in~\cite{Ra94,Ra97}.

In constructing our examples in Theorem~\ref{obs}, we rely on various cases of the strong Artin conjecture. The first of these cases is implicit in the work of Hecke and Maa{\ss},
and subsequent cases were proved by Langlands~\cite{La80}, Tunnell~\cite{Tu81}, Khare--Wintenberger~\cite{KW109,KW209} and Kisin~\cite{Ki09}. (Note that we do not actually need the full force of Khare--Wintenberger and Kisin, and can instead make do with one of the earlier examples available in the literature; see Section~\ref{s4}).\\

We also note a result of Murty--Rajan~\cite{MR96} that is related to strong multiplicity one. For $\pi_1, \pi_2  \in \mathcal{A}_0({\rm GL}_2(\A_\Q))$ with the property that the $L$-functions of the Rankin--Selberg convolutions of ${\rm Sym}^n (\pi_1)$ and ${\rm Sym}^m (\pi_2)$, for $n =m$ and $n = m +2$, are entire, have suitable functional equations, and satisfy GRH, then $ \# \{p \leq x \mid a_p (\pi_1) = a_p (\pi_2)\} = O\left( x ^{5/6 +\epsilon }\right). $\\

The structure of this paper is as follows. In Section~\ref{s1}, we establish notation and recall relevant theorems on the properties of certain types of $L$-functions associated to automorphic representations. In Section~\ref{s2}, we prove the non-dihedral results. In Section~\ref{s3}, we prove the dihedral cases of Theorems~\ref{t2} and~\ref{t3}. Finally, in Section~\ref{s4} we construct the examples from Theorem~\ref{obs}.

\subsection*{Acknowledgements}
The author would like to thank Dinakar Ramakrishnan for his valuable guidance and useful discussions, as well as Farrell Brumley and Abhishek Saha for their helpful comments on earlier drafts of this paper.

\section{Preliminaries} \label{s1}
We begin by introducing some notation.
Let $F$ be a number field and let $S$ be a set of primes in $F$. Then the \emph{lower and upper Dirichlet densities} of $S$ are
\begin{align*}
\underline{\delta}(S) = \lim_{s \rightarrow 1^+} {\rm inf} \frac{\sum_{\mathfrak{p} \in S}{\rm N}\mathfrak{p}^{-s}}{-\log(s-1)}
\end{align*}
and
\begin{align*}
\overline{\delta}(S) = \lim_{s \rightarrow 1^+} {\rm sup} \frac{\sum_{\mathfrak{p} \in S}{\rm N}\mathfrak{p}^{-s}}{-\log(s-1)},
\end{align*}
respectively.

When the lower and upper Dirichlet densities of $S$ coincide, we say that $S$ has a \emph{Dirichlet density}.

\subsection*{Polyhedral automorphic representations}

We will call a cuspidal automorphic representation $\pi$ for GL(2)/$F$  (where $F$ is a number field) \textit{polyhedral} if it corresponds to a 2-dimensional irreducible complex representation $\rho$ of the Weil group $W_F$. This means that the $L$-functions of the two objects are equal: 
\begin{align*}
L(s, \pi) = L(s, \rho).
\end{align*}
Recall that these 2-dimensional irreducible complex representations fall into four different categories, namely \textit{dihedral, tetrahedral, octahedral}, and \textit{icosahedral} (for example, see the Proposition in Section 4.3 of~\cite{Ge97}). An associated cuspidal automorphic representation then inherits the same nomenclature.

Given a 2-dimensional irreducible complex representation $\rho$, the strong Artin conjecture states that it corresponds to a cuspidal automorphic representation for GL(2). This is known when $\rho$ is dihedral (implicit in the work of Hecke and Maa{\ss}), tetrahedral~\cite{La80}, octahedral~\cite{Tu81}, and odd icosahedral~\cite{KW109,KW209,Ki09} (note that a representation $\rho$ is \textit{odd} if ${\rm det}\rho (c) = -1$, where $c$ is complex conjugation).
Note that there were a number of examples of modular odd icosahedral representations already in the literature before the work of Khare--Wintenberger and Kisin. The first example came from Buhler~\cite{Bu78} (which made key use of the correspondence established by Deligne--Serre~\cite{DS74}), which was followed by, amongst others, the work of Buzzard, Dickinson, Shepherd-Barron, and Taylor~\cite{BDST01}, and Taylor~\cite{Ta03}.

Cuspidal automorphic representations of solvable polyhedral type can be characterised by means of their symmetric power lifts (see~\cite{KS00}):
A cuspidal automorphic representation for GL(2) is said to be \textit{dihedral} if it admits a non-trivial self-twist by a (quadratic) character.
It is \textit{tetrahedral} if it is non-dihedral and its symmetric square admits a non-trivial self-twist by a cubic character.
It is \textit{octahedral} if it is not dihedral or tetrahedral and if its symmetric cube admits a non-trivial self-twist by a quadratic character.
Presently known cases of functoriality are not sufficient to determine whether or not icosahedral automorphic representations admit such a description.\\

Given $\pi \in \mathcal{A}_0 ({\rm GL}_n(\A_F))$ there is the associated $L$-function 
\begin{align*}
 L(s, \pi ) = \prod_{ v } L_v (s, \pi),  
\end{align*}
where, for finite $v$ at which $\pi$ is unramified, we have: 
\begin{align*}
L_v (s, \pi) =& {\rm det} \left(I_n - A( \pi _v ) {\rm N } v ^{-s} \right) ^{-1}  \\
=& \prod_{ j=1 }^{ n}  \left( 1 - \alpha _{j, v }{\rm N } v  ^{-s}\right) ^{-1}  .
\end{align*}
The $L$-function $L(s,\pi)$ converges absolutely for ${\rm Re}(s) > 1$, and furthermore it is non-vanishing on $ {\rm Re}(s) = 1 $~\cite{JS76}.

Now let $T$ be the set of all ramified and infinite places. We define the incomplete $L$-function 
\begin{align*}
  L ^ T (s, \pi ) = \prod_{v \not \in T } L_v (s, \pi),
\end{align*}
and the incomplete Dedekind zeta function
\begin{align*}
  \zeta_F^T (s) = \prod_{v \not \in T }(1- {\rm N }v ^{-s})^{-1}  .
\end{align*}

\subsection*{Rankin--Selberg $L$-functions}
Given automorphic representations $\pi$ and $\pi'$ for ${\rm GL}_n(\A_F)$ and ${\rm GL}_m(\A_F)$, respectively,
we have the Rankin--Selberg $L$-function 
\begin{align*}
L(s, \pi \times \pi') = \prod_{ v } L _v (s, \pi \times \pi'),  
\end{align*}
where for finite $v$ at which both $\pi$ and $\pi'$ are unramified, we have: 
\begin{align*}
   L_v (s, \pi \times \pi') = {\rm det} \left(I _{nm} - \left( A_v(\pi) \otimes A_v(\pi') \right){\rm N } v  ^{-s}\right) ^{-1} .
\end{align*}

This global $L$-function converges absolutely for ${\rm Re}(s) > 1$. When $\pi'$ is dual to $\pi$, one knows that $L(s, \pi \times \pi')$ has a simple pole at $s=1$~\cite{JS81}. Otherwise, it is invertible at $s=1$~\cite{Sh81}.\\

Given  $\pi, \pi' \in \mathcal{A}_0({\rm GL}_2(\A_F))$, by~\cite{Ra00}, one knows of the existence of the automorphic tensor product of $\pi$ and $\pi'$, which is an automorphic representation for ${\rm GL}_4(\A_F)$ that we write as $\pi \boxtimes \pi'$. When $\pi$ and $\pi'$ are both dihedral, there is a cuspidality criterion (from~\cite{Ra00}, and later refined in~\cite{Ra04}) which we will make use of later:

For two dihedral automorphic representations 
$\pi, \pi' \in \mathcal{A}_0({\rm GL}_2(\A_F)) $, 
 the automorphic product $\pi \boxtimes \pi'$ is not cuspidal if and only if $\pi$ and $\pi'$ can be induced from the same quadratic extension $K$ of $F$. In this case, let us write $\pi$, $\pi'$ as $I_K^F (\nu), I_K^F (\mu)$, respectively, where $\nu$ and $\mu$ are Hecke characters for $K$. Then 
\begin{align*}
\pi \boxtimes \pi' = I_K^F (\nu \mu) \boxplus I_K^F (\nu \mu^\tau)
\end{align*}
where $\tau$ is the non-trivial element of ${\rm Gal}(K/F)$ and $\mu^\tau$ signifies precomposition with $\tau$.

\begin{remark}
Note that a given dihedral automorphic representation may be induced from more than one quadratic extension. See also Theorem A of~\cite{PR11}.
\end{remark}

\subsection*{Cuspidality of symmetric powers}
For $\pi \in \mathcal{A}_0({\rm GL}_2(\A_F))$, one knows, by Gelbart--Jacquet~\cite{GJ78} and Kim--Shahidi~\cite{KS00}, that the symmetric second and third power representations are isobaric sums of cuspidal automorphic representations.  One also knows that $ {\rm Sym}^2 \pi $ is cuspidal if and only if $\pi$ is non-dihedral and that $ {\rm Sym}^3 \pi $ is cuspidal if and only if $\pi$ is not dihedral or tetrahedral.

\subsection*{Adjoint lift}
Given $\pi \in \mathcal{A}_0({\rm GL}_2(\A_F))$, there exists the adjoint lift ${\rm Ad}(\pi)$ which is an automorphic representation for ${\rm GL}_3(\A_F)$. This lift is in fact isomorphic to ${\rm Sym}^2 (\pi) \otimes \omega ^{-1}$, where $\omega$ is the central character of $\pi$. The adjoint lift (and thus the symmetric square lift) is cuspidal if and only if $\pi$ is non-dihedral~\cite{JS81}.
We make use of this lift to address possible twist-equivalence for cuspidal automorphic representations for GL(2): Given $\pi = \otimes_v' \pi_v$, $\pi' = \otimes_v' \pi'_v \in \mathcal{A}_0(GL_2(\A_F))$, they are twist-equivalent if
 \begin{align*}
{\rm Ad}(\pi_v) \simeq {\rm Ad}(\pi'_v)
\end{align*}
for almost all $v$~\cite{Ra00}.

\section{Proof of Theorems~\ref{t1} and~\ref{t3}(ii)} \label{s2}

Throughout this section, the cuspidal automorphic representations are assumed to be non-dihedral.\\

For the proofs of Theorems~\ref{t1} and~\ref{t3}(ii) we require the following identities of incomplete $L$-functions:

\begin{lemma}\label{s2lem1} Given $\pi,\pi' \in \mathcal{A}_0({\rm GL}_2(\A_F))$ with (unitary) central characters $\omega$, $\omega'$ (respectively), let $T$ be the set of all the infinite places as well as the finite places at which either $\pi$ or $\pi'$ is ramified.
Then we have
  \begin{align*}
    L^T\left(s, \pi \boxtimes \pi \times \overline{\pi} \boxtimes \overline{\pi} \right) =& L^T\left(s, {\rm Ad} (\pi) \times {\rm Ad}(\pi)  \right)  L^T\left(s, {\rm Ad}(\pi) \right)^2   \zeta_F^T(s) \\
    L^T\left(s, \pi \boxtimes \pi \times \overline{\pi} \boxtimes \overline{\pi'}\right) =& L^T\left(s, (\overline{\omega}\otimes \rm{Sym}^3 (\pi)) \times \overline{\pi'} \right) L^T\left(s, \pi \times \overline{\pi'} \right) ^2\\
    L^T\left(s, \pi \boxtimes \pi \times \overline{\pi'} \boxtimes \overline{\pi'}\right) =& L^T \left(s, {\rm Ad}(\pi) \times ({\rm Ad}(\pi') \otimes \overline{\omega}\omega') \right)
L^T\left(s, \omega \overline{\omega'} \otimes \rm{Ad} (\pi) \right) \\
& \cdot  L^T\left(s, \omega \overline{\omega'} \otimes \rm{Ad} (\pi') \right)  L^T \left(s, \omega \overline{\omega '}\right) \\
    L^T\left(s, \pi \boxtimes \overline{\pi} \times \pi' \boxtimes \overline{\pi'}\right) =& L^T \left(s, {\rm Ad}(\pi) \times {\rm Ad}(\pi') \right) L^T\left(s, {\rm Ad}(\pi)\right) L^T\left(s, {\rm Ad}(\pi')\right)  \zeta^T_F(s) .
  \end{align*}
  \end{lemma}

\begin{proof}
This follows from the Clebsch--Gordon decomposition of tensor powers of two-dimensional representations.
Here and in later proofs, we use the fact that the contragredient representation $\widetilde{\pi_v}$ is isomorphic to the complex conjugate representation $\overline{\pi_v}$~\cite{GK75}.

By way of example, we elaborate on the case of $L^T\left(s, \pi \boxtimes \pi \times \overline{\pi'} \boxtimes \overline{\pi'}\right)$. At a finite place $v$ of $F$ where both $\pi$ and $\pi'$ are unramified, we represent the associated Langlands conjugacy classes $A_v(\pi)$ and $A_v(\pi')$ as ${\rm diag}\{\alpha,\beta\}$ and ${\rm diag}\{\gamma,\delta\}$, respectively.
Consider the tensor product $A_v(\pi)\otimes A_v(\pi)$, which we express as
\begin{align*}
     \left( \begin{array}{cc}
     \alpha& \\
     &\beta
     \end{array} \right)
\otimes
     \left( \begin{array}{cc}
     \alpha& \\
     &\beta
     \end{array} \right)
\end{align*}
which is equivalent to
\begin{align*}
\alpha \beta 
\otimes
\left(
     \left( \begin{array}{ccc}
     \alpha / \beta&  &  \\
     & 1 &  \\
     &  & \beta / \alpha
     \end{array} \right)
\oplus
1
\right)
\end{align*}
and we observe that this is a representative of $A_v(\omega \otimes ({\rm Ad}\pi \boxplus 1))$.

For $A_v(\overline{\pi'})\otimes A_v(\overline{\pi'})$, we have 
\begin{align*}
\overline{\gamma \delta}
\otimes
\left(
     \left( \begin{array}{ccc}
     \overline{\gamma}/ \overline{\delta}&  &  \\
     & 1 &  \\
     &  & \overline{\delta} / \overline{\gamma}
     \end{array} \right)
\oplus 1
\right)
\end{align*}
which is a representative of $A_v(\overline{\omega'}\otimes ({\rm Ad}\pi' \boxplus 1))$.

Therefore 
\begin{align*}
A_v(\pi)\otimes A_v(\pi)\otimes A_v(\overline{\pi'})\otimes A_v(\overline{\pi'})
\end{align*}
can be expressed as 
\begin{align*}
A_v\left(\omega \overline{\omega'} \otimes \left(({\rm Ad}\pi \times {\rm Ad}\pi') \boxplus {\rm Ad}\pi \boxplus {\rm Ad}\pi' \boxplus 1\right)\right).
\end{align*}
Observing that
\begin{align*}
& L_v\left(s, \pi \boxtimes \pi \times \overline{\pi'} \boxtimes \overline{\pi'}\right) \\
=& {\rm det}\left(I_{16} - \frac{A_v(\pi)\otimes A_v(\pi)\otimes A_v(\overline{\pi'})\otimes A_v(\overline{\pi'})}{Nv^s}\right)^{-1}\\
=& {\rm det}\left(I_{9} - \frac{\omega \overline{\omega'} \cdot A_v({\rm Ad}\pi)\otimes A_v({\rm Ad}\pi')}{Nv^s}\right)^{-1}
{\rm det}\left(I_{3} - \frac{\omega \overline{\omega'} \cdot A_v({\rm Ad}\pi)} {Nv^s}\right)^{-1}\\
&\cdot {\rm det}\left(I_{3} - \frac{\omega \overline{\omega'} \cdot  A_v({\rm Ad}\pi')}{Nv^s}\right)^{-1}
{\rm det}\left(I_{1} - \frac{\omega \overline{\omega'}}{Nv^s}\right)^{-1}\\
=& L_v(s, \omega \overline{\omega'} \otimes {\rm Ad}\pi \times {\rm Ad}\pi')
L_v(s, \omega \overline{\omega'} \otimes {\rm Ad}\pi)
L_v(s, \omega \overline{\omega'} \otimes {\rm Ad}\pi')
L_v(s, \omega \overline{\omega'})
\end{align*}
we obtain our $L$-function identity.
\end{proof}

\begin{remark}
The purpose of the lemma above is to establish the asymptotic behaviour of various Dirichlet series as real $s \rightarrow 1^+$.
For example, using~\cite{BB11} we observe that, as real $s \rightarrow 1^+$,
\begin{align*}
\log L^T(s, \pi \boxtimes \overline{\pi} \times \pi' \boxtimes \overline{\pi'}) =
\sum \frac{|a_v (\pi)|^2 |a_v (\pi')|^2 }{{\rm N}v ^s} + O\left(1\right),
\end{align*}
and similarly for the other incomplete $L$-functions on the left-hand side of the equations in Lemma~\ref{s2lem1}. 

\end{remark}

\noindent \textbf{Proof of Theorem~\ref{t1}}\\
Let us begin by fixing both $\pi$ and $\pi'$ to be non-tetrahedral (as well as non-dihedral, as mentioned at the beginning of this section) so that we can assume that the symmetric cubes of $\pi$ and $\pi'$ are both cuspidal. At finite places $v$ where both $\pi$ and $\pi'$ are unramified, we denote the trace of $A_v(\pi)$ as $a_v$ and the trace of $A_v(\pi')$ as $b_v$.

By considering the behaviour of the incomplete $L$-functions in Lemma~\ref{s2lem1}, and using the results stated in the previous section on Rankin--Selberg $L$-functions, symmetric powers, and adjoint lifts, we obtain
\begin{align*}
\sum \frac{a_v^i \overline{a_v}^j b_v^k \overline{b_v}^l}{{\rm N}v ^s}= &k(i,j,k,l)\cdot \log\left(\frac{1}{s-1}\right) +o\left(\log\left(\frac{1}{s-1}\right)\right)
\end{align*}
as real $s \rightarrow 1 ^+$, where $k (i,j,k,l)$ is a non-negative integer such that 
\begin{align*}
k (i,j,k,l) \leq 
     \left\{ \begin{array}{cccl}
     0& \text{ for }&(i,j,k,l)=&(2,1,0,1), (0,1,2,1),(1,2,1,0),\\
&&& \text{ and }(1,0,1,2)\\
     2& \text{ for }&(i,j,k,l)=&(1,1,1,1),(2,2,0,0),(0,0,2,2),(2,0,0,2)\\
&&& \text{ and }(0,2,2,0).
     \end{array} \right.
\end{align*}

Note that we are using the fact that the logarithms of incomplete $L$-functions such as those in Lemma~\ref{s2lem1} are asymptotically equivalent (as real $s \rightarrow 1^+$) to the Dirichlet series above.\\

Let $C= C _S $ be the characteristic function of the set $S = \{ v \mid a_v \neq b_v  \}$. We have, for real $s > 1$:
\begin{align}
\sum_{v } \frac{|a _v - b_v|^2 C (v)  }{ {\rm N }v ^{s} } 
\leq  \left(  \sum_{v } \frac{|a _v -b_v|^4 }{ {\rm N }v ^{s} }\right)^{1/2} \left(   \sum_{v \in S} \frac{1 }{ {\rm N }v ^{s} } \right)^{1/2}  \label{s2ineq1}
\end{align}
where the inequality above arises from applying Cauchy--Schwarz.

We divide inequality~(\ref{s2ineq1}) by $\log \left( 1 / (s-1) \right)$, take the limit inferior as $s \rightarrow 1 ^+$, and examine the asymptotic properties of the series. We obtain
\begin{align*}  
  2 &\leq (16) ^{1/2} \cdot \underline{\delta}(S)^{1/2} 
\end{align*}
which gives  
\begin{align*}
  \frac{1}{4} &\leq \underline{\delta}(S).
\end{align*}

Now let us consider the case when at least one of the cuspidal automorphic representations, say $\pi$, is tetrahedral. We apply Theorem 2.2.2 of~\cite{KS02} which states that
\begin{align*}
 {\rm Sym}^3 (\pi) \otimes \omega ^{-1} (\pi) \simeq (\pi \otimes \nu)\boxplus (\pi \otimes \nu^2)
\end{align*}
where $\omega$ is the central character of $\pi$ and $\nu$ is a (cubic) Hecke character that satisfies ${\rm Sym}^2 (\pi) \simeq {\rm Sym}^2 (\pi) \otimes \nu$.
Since
\begin{align*}
L^T(s, \overline{\omega}\otimes {\rm Sym}^3 (\pi) \times \overline{\pi'}) =L^T(s, (\pi \otimes \nu) \times \overline{\pi'}) L^T(s, (\pi \otimes \nu^2) \times \overline{\pi'}),
\end{align*}
the $L$-function on the left-hand side only has a (simple) pole at $s=1$ when $\pi' \simeq \pi \otimes \nu$ or $\pi \otimes \nu^2$.
If this is the case, then $k(i,j,k,l)=1$ for $(i,j,k,l)=(2,1,0,1), (0,1,2,1)$, $(1,2,1,0)$, and $(1,0,1,2)$. This gives us a lower bound of 1/2 for the density of places at which $a_v(\pi) \neq a_v(\pi')$. Thus the lower bound of 1/4 still holds.\\  \qed 

\noindent \textbf{Proof of Theorem~\ref{t3}(ii)}\\
Here the non-dihedral automorphic representations $\pi$ and $\pi'$ are non-twist-equivalent. 

Following a similar approach to before, we obtain
\begin{align*}
\sum \frac{a_v^i \overline{a_v}^j b_v^k \overline{b_v}^l}{{\rm N}v ^s}= &k(i,j,k,l)\cdot \log\left(\frac{1}{s-1}\right) +o\left(\log\left(\frac{1}{s-1}\right)\right)
\end{align*}
as real $s \rightarrow 1 ^+$, where $k(i,j,k,l)$ is a non-negative integer such that
\begin{align*}
k (i,j,k,l) \leq 
     \left\{ \begin{array}{ccl}
     0& \text{ for }&(i,j,k,l)=(2,1,0,1), (0,1,2,1), (1,2,1,0),\text{ and }(1,0,1,2)   \\
     1& \text{ for }&(i,j,k,l)=(1,1,1,1)   \\
     2& \text{ for }&(i,j,k,l)=(2,2,0,0)\text{ and }(0,0,2,2)\\
     c& \text{ for }&(i,j,k,l)=(2,0,0,2)\text{ and }(0,2,2,0),
     \end{array} \right.
\end{align*}
with $c=c (\omega,\omega ')$ equal to zero if $\omega \neq \omega '$ and equal to one otherwise.

Again we divide inequality~(\ref{s2ineq1}) by $\log \left( 1 / (s-1) \right)$, take the limit inferior as $s \rightarrow 1 ^+$, and examine the asymptotic properties of the series given our condition of non-twist-equivalence. We obtain
\begin{align*}  
  2 &\leq (8 +2c) ^{1/2} \cdot \underline{\delta}(S)^{1/2} 
\end{align*}
which implies
\begin{align*}  
 \frac{2}{4+c} &\leq \underline{\delta}(S).
\end{align*}

Thus $\underline{\delta}(S)$ is greater or equal to 2/5 if $\omega = \omega '$, and 1/2 otherwise.\\ \qed

\section{Proof of Theorems~\ref{t2} and~\ref{t3}(i)} \label{s3}

In this section, any pair of cuspidal automorphic representations $\pi$,$\pi'$ will always be non-twist-equivalent.\\

Since the non-dihedral cases have been covered by the proof of Theorem~\ref{t3}(ii) in the previous section, we now need to consider the situation where at least one out of the pair of cuspidal automorphic representations in question is dihedral.
Let us briefly recall our notation for dihedral automorphic representations. Given a dihedral automorphic representation $\pi$ for ${\rm GL}_2(\A_F)$, one knows that it can be induced from a Hecke character $\mu$ of $K$, where $K$ is a suitable quadratic extension of $F$. Then we will denote it as $I_K^F(\mu)$, or $I(\mu)$ if $K$ and $F$ are understood.\\

We split the proof into four cases. The first three cases will account for the situation where both cuspidal automorphic representations are dihedral, and we distinguish between these three cases based on whether a certain property P holds for none, one, or both of the representations. The fourth case addresses the situation when exactly one of the cuspidal automorphic representations is dihedral.

We shall say that a cuspidal automorphic representation $\pi$ has property P if it is dihedral and, furthermore, that if we express $\pi$ as $I_K^F(\mu)$, then the Hecke character $\mu / \mu^{\tau}$ is invariant under $\tau$, the non-trivial element of ${\rm Gal}(K / F)$. Such $\pi$ can be induced from exactly three different quadratic extensions of $F$.\\

We begin with the case of two dihedral automorphic representations that do not satisfy property P.
Here, the general strategy of the proof follows as in the previous section, though the approach to the analysis needs to be altered. The asymptotic properties of the $L$-functions as real $s \rightarrow 1^+$ may be different from before, in that we expect a higher order pole in certain cases. By way of example, we will discuss two particular incomplete $L$-functions in detail, that of $L^T (s, \pi \boxtimes \pi \times \overline{\pi'}\boxtimes \overline{\pi'})$ and $L^T (s, \pi \boxtimes \pi \times \overline{\pi}\boxtimes \overline{\pi})$.

With regard to the first incomplete $L$-function, we shall explain why it has a pole of order at most two.
If $\pi$, $\pi'$ cannot be induced from the same quadratic extension, then the automorphic tensor product $\pi \boxtimes \overline{\pi'}$ is cuspidal, and so 
$L^T(s, \pi \boxtimes \overline{\pi'} \times \pi \boxtimes \overline{\pi'})$
has at most a pole of order one.
If $\pi, \pi'$ can be induced from the same quadratic extension, then let us denote this extension as $K$. We write $\pi =I_K^F(\nu)=I(\nu)$ and $\pi' =I_K^F(\mu)=I(\mu)$, where $\mu, \nu$ are Hecke characters of $K$. We obtain
\begin{align*}
\pi \boxtimes \overline{\pi'} \simeq I(\nu \overline{\mu}) \boxplus I(\nu \overline{\mu}^\tau)
\end{align*}
and thus 
\begin{align*}
L^T(s, \pi \boxtimes \overline{\pi'} \times \pi \boxtimes \overline{\pi'})
=& L^T(s, I(\nu \overline{\mu}) \times  I(\nu \overline{\mu}))
 L^T(s, I(\nu \overline{\mu}) \times  I(\nu \overline{\mu}^\tau))^2\\
& \cdot L^T(s, I(\nu \overline{\mu}^\tau) \times  I(\nu \overline{\mu}^\tau)).
\end{align*}
If neither $I(\nu \mu)$ nor $I(\nu \mu ^\tau)$ is self-dual, then the right-hand side has a pole of order at most two.
Now let us consider the case when exactly one of them is self-dual. We note that the middle $L$-function on the right-hand side will contribute to the order of the pole of the overall expression if and only if $I(\nu \mu)$ and $I(\nu \mu ^\tau)$ are dual. However, this is not possible given that one is self-dual and the other is not.
Lastly, we have the case when both are self-dual, in which case the middle expression contributes to the order of the pole if and only if $I(\nu \mu) \simeq I(\nu \mu ^\tau)$. This means that either $\nu =\nu ^\tau$ or $\mu =\mu ^\tau$, implying that either $I(\nu)$ or $I(\mu)$ is not cuspidal, in contradiction to our original assumption.

Note that in the analysis above, the issue of whether $\nu / \nu^\tau$ (or $\mu / \mu^\tau$) is invariant under ${\rm Gal}(K/ F)$ did not arise. This is not the case for the next incomplete $L$-function that we address, where we need to make the assumption that both  $\nu / \nu^\tau$ and $\mu / \mu^\tau$ are \textit{not} ${\rm Gal}(K/ F)$-invariant.

We now consider $L^T(s, \pi \boxtimes \pi \times \overline{\pi} \boxtimes \overline{\pi})$ and show that, under our assumption, it can only have a pole of order at most three. Note that
\begin{align*}
\pi \boxtimes \overline{\pi} &\simeq I(1) \boxplus I(\nu / \nu^\tau)\\
& \simeq 1 \boxplus \chi \boxplus I(\nu / \nu^\tau),
 \end{align*}
where $\chi$ is the quadratic Hecke character associated to $K / F$. Since we have assumed that $\nu / \nu^\tau$ is not ${\rm Gal}(K / F)$-invariant, we know that $I(\nu / \nu^\tau)$ is cuspidal. Thus 
\begin{align*}
L^T(s, \pi \boxtimes \overline{\pi} \times \pi \boxtimes \overline{\pi})=&\zeta^T(s)L^T(s, \chi)^2 L^T(s, \chi \times \chi)L^T(s, I(\nu / \nu^\tau))\\
&\cdot L^T(s, \chi \times I(\nu / \nu^\tau))L^T(s, I(\nu / \nu^\tau) \times I(\nu / \nu^\tau)).
\end{align*}
On the right-hand side, the first and third $L$-functions have simple poles at $s=1$. The last $L$-function has a simple pole at $s=1$ if $I(\nu / \nu^\tau)$ is self-dual, otherwise it is invertible at that point. The remaining $L$-functions on the right-hand side are all invertible at $s=1$. Therefore the $L$-function on the left-hand side has a pole of order at most three.

Taking a similar approach to the analysis of the remaining incomplete $L$-functions, we obtain
\begin{align*}
\sum \frac{a_v^i \overline{a_v}^j b_v^k \overline{b_v}^l}{{\rm N}v ^s}= &k(i,j,k,l)\cdot \log\left(\frac{1}{s-1}\right) +o\left(\log\left(\frac{1}{s-1}\right)\right)
\end{align*}
as real $s \rightarrow 1 ^+$, where $k (i,j,k,l)$ is a non-negative integer such that 
\begin{align*}
k (i,j,k,l) \leq 
     \left\{ \begin{array}{ccll}
     1& \text{ for }&(i,j,k,l)=&(2,1,0,1), (0,1,2,1), (1,2,1,0),\\
     &&&\text{ and }(1,0,1,2)   \\
     2& \text{ for }&(i,j,k,l)=&(1,1,1,1), (2,0,0,2),\text{ and }(0,2,2,0)\\
     3& \text{ for }&(i,j,k,l)=&(2,2,0,0)\text{ and }(0,0,2,2).
     \end{array} \right.
\end{align*}

Continuing with our proof, we divide inequality~(\ref{s2ineq1}) by $\log \left( 1 / (s-1) \right)$, take the limit inferior as $s \rightarrow 1 ^+$, and examine the asymptotic properties of the series given our non-twist-equivalent condition. We obtain
\begin{align*}  
  2 &\leq (18 ) ^{1/2} \cdot \underline{\delta}(S)^{1/2} 
\end{align*}
which gives
\begin{align*}  
\frac{2}{9} & \leq \underline{\delta}(S).
\end{align*}
Thus $\pi$ and $\pi'$ differ locally at a set of places of lower Dirichlet density at least 2/9.\\

We now consider the case where both dihedral automorphic representations satisfy property P.
Again we write $\pi =I_K^F(\nu)$ and $\pi' =I_K^F(\mu)$, where $\mu, \nu$ are Hecke characters for $K$; this time both $\nu / \nu^\tau$ and $\mu / \mu^\tau$ are ${\rm Gal}(K/ F)$-invariant.

Note that
\begin{align*}
\pi \boxtimes \overline{\pi} \simeq & I(1) \boxplus I\left(\frac{\nu}{\nu^\tau}\right)\\
\simeq & 1 \boxplus \chi \boxplus \frac{\nu}{\nu^\tau} \boxplus \left(\frac{\nu}{\nu^\tau} \cdot \chi \right).
\end{align*}
We also know that  $\pi \boxtimes \overline{\pi} \simeq 1 \boxplus {\rm Ad}\pi$, so 
\begin{align*}
{\rm Ad}\pi \simeq \chi \boxplus \frac{\nu}{\nu^\tau} \boxplus (\frac{\nu}{\nu^\tau} \cdot \chi )
\end{align*}
and similarly, 
\begin{align*}
{\rm Ad}\pi' \simeq \chi \boxplus \frac{\mu}{\mu^\tau} \boxplus (\frac{\mu}{\mu^\tau} \cdot \chi ).
\end{align*}

We now want to consider the lower Dirichlet density of the set 
\begin{align*}
S:=\{v \mid a_v({\rm Ad}\pi) \neq a_v({\rm Ad}\pi') \},
\end{align*}
as this will give us a lower bound for the lower Dirichlet density of the set 
\begin{align*}
S':=\{v \mid a_v(\pi) \neq a_v(\pi') \}.
\end{align*}
We first determine the decompositions of the relevant incomplete $L$-functions.

\begin{lemma}
Given $\pi =I_K^F(\nu)$ and $\pi' =I_K^F(\mu)$, where $\mu$ and $\nu$ are Hecke characters for $K$ such that both $\nu / \nu^\tau$ and $\mu / \mu^\tau$ are ${\rm Gal}(K/F)$-invariant, we have 
\begin{align*}
 L^T\left(s, {\rm Ad}\pi \times {\rm Ad}\pi\right)=&\zeta^T (s)
 L^T\left(s, \frac{\nu}{\nu^\tau}\cdot \chi\right)^2
 L^T\left(s,\frac{\nu}{\nu^\tau}\right)^2
 L^T\left(s, \frac{\nu}{\nu^\tau} \cdot \frac{\nu}{\nu^\tau}\right)^2\\
& \cdot L^T\left(s, \frac{\nu}{\nu^\tau} \cdot \frac{\nu}{\nu^\tau} \cdot \chi\right)^2\\
 L^T\left(s, {\rm Ad}\pi \times {\rm Ad}\pi'\right)=&\zeta ^T (s)
 L^T\left(s, \frac{\nu}{\nu^\tau}\cdot \chi\right)
 L^T\left(s, \frac{\nu}{\nu^\tau}\right)
 L^T\left(s, \frac{\mu}{\mu^\tau}\cdot \chi\right)\\
& \cdot L^T\left(s,\frac{\nu}{\nu^\tau} \cdot \frac{\mu}{\mu^\tau}\right)^2
 L^T\left(s, \frac{\nu}{\nu^\tau} \cdot \frac{\mu}{\mu^\tau} \cdot \chi\right)^2
 L^T\left(s, \frac{\mu}{\mu^\tau}\right)\\
 L^T\left(s, {\rm Ad}\pi' \times {\rm Ad}\pi'\right)=&\zeta^T (s)
 L^T\left(s, \frac{\mu}{\mu^\tau}\cdot \chi\right)^2
 L^T\left(s,\frac{\mu}{\mu^\tau}\right)^2
 L^T\left(s, \frac{\mu}{\mu^\tau} \cdot \frac{\mu}{\mu^\tau}\right)^2\\
& \cdot L^T\left(s, \frac{\mu}{\mu^\tau} \cdot \frac{\mu}{\mu^\tau} \cdot \chi\right)^2.\\
\end{align*}
\end{lemma}
\begin{proof}
As in the proof for Lemma~\ref{s2lem1}.
\end{proof}

Since $\nu / \nu^\tau$ and $\mu / \mu^\tau$ are ${\rm Gal}(K / F)$-invariant, we deduce that they are quadratic characters. Neither of them is equal to $\chi$ as this would imply that 
\begin{align*}
I(\nu) \simeq \nu \boxplus (\nu \cdot \chi),
\end{align*}
or the same for $\mu$, and thus either $\pi$ or $\pi'$ would not be cuspidal.
So $L^T(s, {\rm Ad}\pi \times {\rm Ad}\pi)$ and $L^T(s, {\rm Ad}\pi' \times {\rm Ad}\pi')$ have a pole of order three at $s=1$. We also note that $\nu / \nu^\tau$ cannot be equal to $\mu / \mu^\tau$ or $\chi \cdot \mu / \mu^\tau$ otherwise $\pi$ and $\pi'$ would have adjoint lifts that agree locally almost everywhere and thus be twist-equivalent, in contradiction to the assumption made at the beginning of this section. Therefore, $L^T(s, {\rm Ad}\pi \times {\rm Ad}\pi')$ has a pole of order one at $s=1$.

We apply this to the following inequality which has been obtained via the application of the Ramanujan bound (which is known to hold for dihedral representations)
\begin{align*}
\sum_{v \not \in T}\frac{|a_v ({\rm Ad}\pi) -a_v({\rm Ad}\pi')|^2}{Nv^s}\leq 16 \sum_{v \in S}\frac{1}{Nv^s}.
\end{align*}
Dividing by $\log(1/ (s-1))$, taking the limit inferior as $s \rightarrow 1^+$, and examining the asymptotic properties of the resulting inequality, we obtain $1/4 \leq \underline{\delta}(S)$. The same bound then holds for the lower Dirichlet density of the set $S'$, which is sufficient to claim that $\pi$ and $\pi'$ differ locally for a set of places of lower Dirichlet density greater than 2/9.\\

We move on to the case where exactly one of the dihedral automorphic representations satisfies property P, and we use the same notation as before. Without loss of generality, let us say that property P holds for $\pi = I(\nu)$ (i.e., $\nu / \nu^\tau = (\nu / \nu^\tau)^\tau$).
We will compare $\pi = I(\nu)$ and $\pi' = I(\mu)$ and show that $a_v(\pi)$ and $a_v(\pi')$ differ at a set of places of density at least 1/4.

Note that for any place $w$ of $K$, we have that $\nu (w) =\pm \nu^\tau (w)$. For a place $v$ of $F$, this implies $|a_v (I(\nu))|=0 \text{ or }2$. On the other hand, $\mu / \mu^\tau \neq (\mu / \mu^\tau)^\tau$, so $\mu ^2 / (\mu^\tau)^2 $ is not the trivial character.
Then there exists a set of density 1/4 of places $v$ of $F$ which split in $K$ and at which $\mu^2 / (\mu^\tau)^2 \neq 1$ for $w \mid v$. Thus $\mu(w) \neq \pm \mu^\tau (w)$ and so $|a_v(I(\mu))| \neq 0 \text{ or }2$.
 Therefore, $I(\nu)$ and $I(\mu)$ differ locally at a set of places of density at least 1/4.\\

Lastly, we address the case where exactly one of the cuspidal automorphic representations is not dihedral.
This follows in a similar manner to the proof of Theorem~\ref{t3}(ii) in the previous section (as well as the first case of this proof). We will elaborate on this for $L^T(s,\pi \boxtimes \pi \times \overline{\pi'}\boxtimes \overline{\pi'})$.

Let us say that $\pi$ is non-dihedral and $\pi' = I(\mu)$. 
Then
\begin{align*}
\overline{\pi'}\boxtimes \overline{\pi'}
&\simeq I(\overline{\mu}^2) \boxplus I(\overline{\mu \mu^\tau})\\
& \simeq I(\overline{\mu}^2) \boxplus \overline{\mu \mu^\tau} \boxplus \overline{\mu \mu^\tau} \cdot \chi.
\end{align*}
Note that $I(\overline{\mu}^2)$ may or may not be cuspidal, depending on the properties of $\mu$. One knows that
\begin{align*}
\pi \boxtimes \pi \simeq {\rm Sym}^2 \pi \boxplus \omega,
\end{align*}
where $\omega$ is the central character of $\pi$, and ${\rm Sym}^2 \pi$ is cuspidal.
Making use of these, we find that 
\begin{align*}
L^T(s,\pi \boxtimes \pi \times \overline{\pi'}\boxtimes \overline{\pi'})
= & L^T(s, {\rm Sym}^2 \pi \times I(\overline{\mu}^2)) L^T(s, {\rm Sym}^2 \pi \times  \overline{\mu \mu^\tau})\\
& \cdot  L^T(s, {\rm Sym}^2 \pi \times  \overline{\mu \mu^\tau}\cdot \chi)
 L^T(s, \omega \times I(\overline{\mu}^2))\\
&\cdot  L^T(s, \omega \times  \overline{\mu \mu^\tau}) L^T(s, \omega \times  \overline{\mu \mu^\tau}\cdot \chi).
\end{align*}

The first three $L$-functions on the right-hand side will be invertible at $s=1$. The fourth $L$-function will either be invertible at $s=1$ or have a simple pole there.
At most one of the last two $L$-functions (on the right-hand side) can have a (simple) pole at $s=1$. So the $L$-function on the left-hand side has a pole of order at most two.

The analysis of the remaining quadruple-product $L$-functions follows in a similar way, and we obtain:
\begin{align*}
\sum \frac{a_v^i \overline{a_v}^j b_v^k \overline{b_v}^l}{{\rm N}v ^s}= &k(i,j,k,l)\cdot \log\left(\frac{1}{s-1}\right) +o\left(\log\left(\frac{1}{s-1}\right)\right)
\end{align*}
as real $s \rightarrow 1 ^+$, where $k (i,j,k,l)$ is a non-negative integer such that 
\begin{align*}
k (i,j,k,l) \leq 
     \left\{ \begin{array}{ccll}
     0& \text{ for }&(i,j,k,l)=&(2,1,0,1), (0,1,2,1), (1,2,1,0),\\
     &&&\text{ and }(1,0,1,2)   \\
     1& \text{ for }&(i,j,k,l)=&(1,1,1,1)\\
     2& \text{ for }&(i,j,k,l)=&(2,2,0,0), (2,0,0,2),\text{ and }(0,2,2,0)\\
     4& \text{ for }&(i,j,k,l)=&(0,0,2,2).
     \end{array} \right.
\end{align*}

Divide inequality~(\ref{s2ineq1}) by $\log \left( 1 / (s-1) \right)$ and take the limit inferior as $s \rightarrow 1 ^+$. Then
\begin{align*}  
  2 &\leq (14) ^{1/2} \cdot \underline{\delta}(S)^{1/2} 
\end{align*}
giving
\begin{align*}  
\frac{2}{7} & \leq \underline{\delta}(S),
\end{align*}
which completes the proof of Theorem~\ref{t3}.

\section{Examples}\label{s4}
 
We begin by constructing examples to establish that Theorems~\ref{t1},~\ref{t2}, and~\ref{t3}(ii) are sharp. We then describe the example of Serre which demonstrates that Ramakrishnan's refinement of strong multiplicity one~\cite{Ra94} is also sharp. We finish by addressing the remaining examples from Theorem~\ref{obs}.

\subsection{An octahedral example for Theorem~\ref{t1}.\\}\label{s4-1}

We prove the existence of a pair of non-isomorphic octahedral automorphic representations that agree locally for a set of places of density 3/4, proving that Theorem~\ref{t1} is sharp.

Let $\widetilde{S_4}$ denote the binary octahedral group, which is a double cover of the octahedral group $S_4$ and has presentation
\begin{align*}
<\alpha,\beta,\gamma \mid \alpha ^4 =\beta ^3 =\gamma ^2 =-1>.
\end{align*}
The conjugacy classes are $[1],[-1],[\alpha ^i],[\beta ^j],[\gamma]$, where $i =1,2,3$ and $j =1,2$. They have sizes 1, 1, 6, 8, and 12, respectively.

Two of the 2-dimensional complex irreducible representations of the binary octahedral group have the following character table:
    \begin{center}
      \begin{tabular}{r|c|c|c|c|c|c|c|c} 
      &$[1]$&$[-1]$&$[\alpha]$&$[\alpha^2]$& $[\alpha^3]$&$[\beta]$&$[\beta^2]$&$[\gamma]$ \\ \hline
 $\eta$ &2&$-2$&$\sqrt{2}$&0&$-\sqrt{2}$&1&$-1$&0 \\
 $\eta'$ &2&$-2$&$-\sqrt{2}$&0&$\sqrt{2}$&1&$-1$&0 \\
      \end{tabular}
    \end{center}

Given number fields $K$ and $F$ such that ${\rm Gal}(K/ F) \simeq \widetilde{S_4}$ and given the complex representations $\rho, \rho'$ associated to this Galois group which arise from $\eta, \eta'$ (respectively), we apply the Chebotarev density theorem to determine that 
\begin{align*}
{\rm tr}\rho ({\rm Frob}_v) = {\rm tr}\rho' ({\rm Frob}_v)
\end{align*}
holds for a set of finite places $v$ of density exactly 3/4.
Applying the Langlands--Tunnell theorem (see~\cite{Tu81}), we obtain the corresponding cuspidal automorphic representations $\pi$ and $\pi'$ with the required properties.

\subsection{A dihedral example for Theorem~\ref{t2}.\\} \label{s4-2}

We shall establish the existence of a pair of dihedral cuspidal automorphic representations that are not twist-equivalent but agree locally for a set of places of density 7/9.

Consider two number fields $K, K '$ which are $S_3$-extensions of a number field $F$ such that $K \cap K '$ is a degree 2 (Galois) extension of $F$ (for example, $K =\Q (\zeta_3, \sqrt[3]{5})$, $K' =\Q (\zeta_3, \sqrt[3]{7})$, and $F =\Q$).
Note that $S_3$ has a complex irreducible representation of degree 2 that has the following character table:

\begin{center}
      \begin{tabular}{r|c|c|c} 
        &$[(1)]$&$[(12)]$&$[(123)]$ \\ \hline
 $\tau$&2&0&$-1$ \\
      \end{tabular}
\end{center}
Fixing isomorphisms 
\begin{align*}
{\rm Gal}(K/F)\simeq S_3 \simeq {\rm Gal}(K'/F)
\end{align*}
we obtain two dihedral Artin representations $\rho, \rho'$.

We now establish the density of the set of primes at which the respective Frobenius conjugacy classes are equal and determine that the two representations are not twist-equivalent. Since 
\begin{align*}
{\rm Gal}(K K'/ F)\simeq \{(\phi,\psi) \in {\rm Gal}(K/F)\times {\rm Gal}(K'/F)\mid \phi |_{K \cap K'} = \psi  |_{K \cap K'} \}
\end{align*}
and 
\begin{align*}
{\rm Frob}_{v, K K' / F}|_{K^{(')}} ={\rm Frob}_{v, K^{(')} / F}
\end{align*}
we apply the Chebotarev density theorem to determine the density of places at which pairs of the form
$({\rm tr}\rho ({\rm Frob}_{v, K / F}), {\rm tr}\rho ({\rm Frob}_{v, K' / F}))$
occur:

\begin{center}
      \begin{tabular}{r|c|c|c} 
&$\phi,\psi$&$({\rm tr}\rho ({\rm Frob}_{v, K / F}), {\rm tr}\rho ({\rm Frob}_{v, K' / F}))$ &density\\ \hline
&(1),(1)&(2,2)&1/18\\
&(1),(123)&(2,$-1$)&2/18\\
&(123),(1)&($-1$,2)&2/18\\
&(123),(123)&($-1$,$-1$)&4/18\\
&(12),(12)&(0,0)&9/18\\
      \end{tabular}
\end{center}

Examining the table, we conclude that the respective traces of Frobenius agree exactly for a set of places of density $(1+4+9)/18=7/9$.
The strong Artin conjecture for solvable groups~\cite{AC89} then implies the existence of dihedral automorphic representations $\pi, \pi'$ that agree locally for a set of density 7/9.

Now we determine that the automorphic representations are not twist-equivalent. Observe that, for a positive density of places $v$, both
\begin{align*}
a_v(\pi) &= 2\\
a_v(\pi') &= -1
\end{align*}
hold. The quotient of their absolute values is not equal to 1, and so $\pi$ and $\pi'$ cannot differ by a unitary Hecke character.
Therefore, we can conclude that Theorem~\ref{t2} is sharp.

\subsection{An icosahedral example for Theorem~\ref{t3}(ii).\\} \label{s4-3}

We construct a pair of non-twist-equivalent (odd) icosahedral automorphic representations whose Hecke eigenvalues are equal at a set of places of density 3/5, which will imply that Theorem~\ref{t3}(ii) is sharp.

To address the structure of icosahedral Artin representations, we make use of the following proposition from~\cite[Proposition 2.1]{Wa03}:
\begin{proposition}[Wang] \label{prop-wang}
Let 
\begin{align*}
r: {\rm Gal}(\overline{\Q}/ \Q) \rightarrow {\rm GL}_2(\C)
\end{align*}
be an icosahedral representation, and denote its image as $G$.
Then 
\begin{align*}
G \simeq (\widetilde{A_5}\times \mu_{2m})/ \pm (1,1)
\end{align*}
where $\widetilde{A_5}$ is the binary icosahedral group (which is isomorphic to ${\rm SL}_2 \F_5$), and $\mu_{2m}$ is the group of roots of unity of order $2m$.

An irreducible representation $\rho$ of $G$ can be decomposed uniquely into a pair $(\rho_0,\chi)$ where $\rho_0 =\rho |_{\widetilde{A_5}}$ is an irreducible representation of $\widetilde{A_5}$, $\chi = \rho |_{\mu_{2m}}$ is a (1-dimensional) representation of $\mu_{2m}$, and we have the condition
\begin{align*}
\rho_0 (-1) \chi (-1) =I.
\end{align*}
Every such pair of irreducible representations gives an irreducible representation of $G$.
\end{proposition}

Let us now consider the 2-dimensional irreducible representations for the binary icosahedral group $\widetilde{A_5}$. This group has presentation
\begin{align*}
<\alpha,\beta,\gamma \mid \alpha ^5 =\beta ^3 =\gamma ^2 =-1>.
\end{align*}
The conjugacy classes are $[1],[-1],[\alpha^i],[\beta^j],$ and $[\gamma]$, where $i =1,2,3,4$ and $j =1,2$. Their sizes are 1, 1, 12, 20, and 30, respectively.
The binary icosahedral group has exactly two complex irreducible representations of dimension 2, with character table:
    \begin{center}
      \begin{tabular}{r|c|c|c|c|c|c|c|c|c} 
      &$\{1\}$&$\{-1\}$&$[\alpha]$&$[\alpha^2]$& $[\alpha^3]$&$[\alpha^4]$&$[\beta]$&$[\beta^2]$&$[\gamma]$ \\ \hline
 $\eta$&2&$-2$&$c$&$-c '$&$c'$&$-c$&1&$-1$&0 \\
 $\eta'$&2&$-2$&$c'$&$-c$&$c$&$-c'$&1&$-1$&0 \\
      \end{tabular}
    \end{center}
where
\begin{align*}
c =\frac{1 +\sqrt{5}}{2} \text{ and } c' =\frac{1-\sqrt{5}}{2}.
\end{align*}

Fix an Artin representation
\begin{align*}
\tau: {\rm Gal}(\overline{\Q}/ \Q) \rightarrow {\rm GL}_2(\C)
\end{align*}
that is of odd icosahedral type. 
Now Proposition~\ref{prop-wang} implies that $\tau$ factors through a finite Galois group $G$ which is isomorphic to
\begin{align*}
(\widetilde{A_5}\times \mu_{2m})/ \pm (1,1)
\end{align*}
for some $m$.

Decomposing the representation into the pair $(\tau_0,\chi)$ of irreducible representations on $\widetilde{A_5}$ and $\mu_{2m}$ (respectively), we note that $\tau_0$ must be isomorphic to either $\eta$ or $\eta'$ from the character table above. Let $\tau_0'$ correspond to the other representation. We construct another icosahedral Artin representation $\tau'$, corresponding to the pair $(\tau_0',\chi)$, which is also odd. We then apply the Chebotarev density theorem, which implies that
\begin{align*}
{\rm tr\ }\tau ({\rm Frob}_v)={\rm tr\ }\tau' ({\rm Frob}_v)
\end{align*}
holds for a set of places $v$ of density 3/5.

To these two representations $\tau$ and $\tau'$ we apply the strong Artin conjecture for odd icosahedral representations (due to Khare--Wintenberger~\cite{KW109,KW209} and Kisin~\cite{Ki09}) and obtain corresponding icosahedral automorphic representations, which we denote as $\pi$ and $\pi'$, respectively.

We now determine that they are not twist-equivalent. There exist a positive density of places $v$ where
\begin{align*}
a_v(\pi) &= c \cdot \chi(\alpha)\\
a_v(\pi') &= c' \cdot \chi(\alpha)
\end{align*}
both hold simultaneously.
The quotient of their absolute values is not 1, so $\pi$ and $\pi'$ cannot differ by a unitary Hecke character.\\

As mentioned earlier, we point out that we do not need to rely on the work of Khare--Wintenberger and Kisin in order to produce a suitable pair of icosahedral representations. Instead we can make use of one of the earlier examples of (odd) icosahedral representations, such as the one whose modularity is addressed by~\cite{BDST01} and is described in~\cite{JM}. This representation $\rho$ is associated to the polynomial (which describes the splitting field) $x ^5 - 8 x ^3 - 2 x ^2 + 31x + 74$, has conductor 3547, and odd quadratic determinant. It corresponds to a weight one modular form that we will denote as $f = f (\rho)$.

As before, we write $\rho$ as $(\rho_0,\chi)$ and construct another icosahedral Artin representation $\rho'$ which corresponds to $(\rho_0',\chi)$, where $\rho_0$ corresponds to one of the complex irreducible 2-dimensional representations $\eta$ or $\eta'$ in the character table above, and $\rho_0'$ corresponds to the other. Now ${\rm tr}(\rho ({\rm Frob}_p)) \in \Q (\sqrt{5})$, and this new representation $\rho'$ can also be constructed by pre-composing $\rho$ with the non-trivial automorphism $\tau$ of ${\rm Gal}(\Q (\sqrt{5})/\Q)$. On the modular side, $f ^\tau$ is the weight one icosahedral modular form that corresponds to $\rho'$. By the same reasoning as earlier, we note that $f$ and $f^\tau$ are not twist-equivalent. Applying the Chebotarev density theorem, we see that the set of primes at which $f$ and $f ^\tau$ have equal Hecke eigenvalues has density 3/5. 

\subsection{An example of Serre.\\} \label{s4-4}

Here we briefly describe the dihedral example of J.-P.~Serre~\cite{Se77}
which demonstrates that D.~Ramakrishnan's refinement of strong multiplicity one~\cite{Ra94} is sharp.

The quaternion group $Q_8$ has a unique 2-dimensional complex irreducible representation, which we will denote as $\tau$.
Consider the group
\begin{align*}
G = Q_8 \times \{\pm 1\}
\end{align*}
 and define two representations of $G$, denoted by $\rho $ and $\rho '$, as $ \tau  \otimes 1$ and $ \tau  \otimes {\rm sgn}$, respectively. We note that both  $\rho$ and $\rho'$ are irreducible.
The quaternion group, and thus $G$, are known to appear as a Galois group of a finite extension of number fields. Therefore, any representation of $G$ can be lifted to a representation of ${\rm Gal}(\overline{\Q}/\Q)$.

Now let $S$ be the subset $\{(+1, -1),(-1,-1)\}$ of $G$. The traces of $\rho$ and $\rho'$ agree exactly outside $S$, and note that $ |S | / |G | = 2/ 2 ^{4} = 1 /8$.
Since $G$ is nilpotent, by Arthur--Clozel~\cite{AC89} the strong Artin conjecture holds for $\rho$ and $\rho'$, so there exist corresponding (dihedral) cuspidal automorphic representations $\pi$ and $\pi'$.
One concludes that D.~Ramakrishnan's theorem is sharp.

\subsection{Further examples.\\}

We construct the three remaining examples from Theorem~\ref{obs}.

The twist-equivalent icosahedral pair of automorphic representations can be constructed as follows:
Consider the icosahedral cuspidal automorphic representation $\pi$ from subsection~\ref{s4-3} and let $\chi$ be a quadratic Hecke character. 
Using the character table from that subsection and the Chebotarev density theorem, we see that the pair $\pi$ and $\pi \otimes \chi$ agree locally exactly at a set of places of density 5/8.\\

For the construction of the two examples of pairs of tetrahedral representations in Theorem~\ref{obs}, we will need to make use of the character table for the binary tetrahedral group.
First, define $i,j,$ and $k$ to be the quaternions and let $\omega = - \frac{1}{2} (1 +i +j +k)$ (note that $\omega$ has order 3). Then the binary tetrahedral group is generated by $ i,j, \omega $.
The character table for the irreducible representations of dimensions one and two is:
    \begin{center}
      \begin{tabular}{r|c|c|c|c|c|c|c} 
        &$\{1\}$&$\{-1\}$&$[i]$&$[ \omega ]$&$[-\omega]$&$[\omega ^2]$&$[-\omega ^2]$ \\ \hline
 $\chi _0$&1&1&1&1&1&1&1\\
 $\chi _1$&1&1&1&$\zeta$&$\zeta$&$\zeta ^2$&$\zeta ^2$ \\
 $\chi _2$&1&1&1&$\zeta ^2$&$\zeta ^2$&$\zeta$&$\zeta$\\
 $\rho$&2&$-2$&0&$-1$&1&$-1$&1\\
 $\rho \otimes \chi _1$ &2&$-2$&0&$-\zeta$&$\zeta$&$-\zeta ^2$&$\zeta ^2$ \\
 $\rho \otimes \chi _2$ &2&$-2$&0&$-\zeta ^2$&$\zeta ^2$&$-\zeta$&$\zeta$\\
      \end{tabular}
    \end{center}
where $\zeta = e ^{2 \pi i /3}. $ 
Note that the conjugacy class $[i]$ has size 6, and the conjugacy classes $[\omega]$, $[-\omega]$, $[\omega ^2]$, and $[-\omega ^2]$ all have size 4.
 
Since $\rho$ has tetrahedral image in ${\rm PGL}_2(\C)$, it corresponds, by Langlands~\cite{La80}, to a tetrahedral automorphic representation $\pi$ for GL(2). If we twist $\pi$ by a quadratic Hecke character $\chi$, we note that $\pi $ and $ \pi \otimes \chi$ agree locally at a set of places of density 5/8.\\

We move on to the example of the non-twist-equivalent pair of tetrahedral representations. Now an Artin representation of tetrahedral type which factors through a binary tetrahedral Galois group ${\rm Gal}(K/ F)$ has the structure
\begin{diagram}
K\\
\dLine_{2}\\
L\\
\dLine_{4}\\
k\\
\dLine_{3}\\
F
\end{diagram}
where $L / F$ and $k / F$ are Galois extensions with groups isomorphic to $A_4$ and $\Z/3\Z$, respectively.

We consider two tetrahedral representations $\rho_1 $ and $ \rho_2$ such that they factor through Galois groups ${\rm Gal}(K_1/ F) $ and $ {\rm Gal}(K_2/ F)$, respectively, with structure
 \begin{diagram}
 K_1&&&&K_2 \\
\dLine &&&& \dLine \\
 L_1&&&&L_2 \\
 &\rdLine && \ldLine & \\
& &k&& \\
& &\dLine&& \\
&& F&&
 \end{diagram}
where $L_1 \neq L_2$ and $K_1 \neq K_2$.

We now need to establish the Galois group of the compositum $K_1 K_2$. 
We make use of the embedding
\begin{align*}
{\rm Gal}(K_1K_2/F) &\hookrightarrow {\rm Gal}(K_1/F) \times {\rm Gal}(K_2/F)\\
  \sigma & \mapsto (\sigma|_{K_1} , \sigma |_{K_2} ),
\end{align*}
which has image
\begin{align*}
H = \{ (\phi , \psi ) \mid \phi |_{K_1 \cap K_2} , \psi  |_{K_1 \cap K_2}\}.
\end{align*}

Expressing ${\rm Gal}(K_1 K_2 / F)$ as a subset of ${\rm Gal}(K_1/ F) \times {\rm Gal}(K_2/ F)$, we see that it consists exactly of:
\begin{itemize}
 \item pairs of the form $(a,b)$ where $ a,b \in \{ \pm 1, \pm i, \pm j, \pm k\}$,
 \item pairs of the form $(a,b)$ where $ a,b \in \{ \pm \omega , \pm i\omega, \pm j\omega, \pm k\omega\}$,
 \item pairs of the form $(a,b)$ where $ a,b \in \{ \pm \omega^2 , \pm i\omega^2, \pm j\omega^2, \pm k\omega^2\}$.
\end{itemize}

Since  
\begin{align*}
{\rm Frob}_{v, K_1 K_2 / F}|_{K_i} ={\rm Frob}_{v, K_i / F}
\end{align*}
for $i = 1,2$ (as mentioned earlier), we now just need to count the occurrences of pairs whose components have the same trace. These are:
\begin{itemize}
 \item pairs $(1,1)$ and $ (-1,-1)$
 \item pairs $(a,b)$ where $ a,b \in \{\pm i, \pm j, \pm k\}$
 \item pairs $(a,b)$ where $ a,b \in \{\omega, -i \omega, -j \omega, -k \omega \}$
 \item pairs $(a,b)$ where $ a,b \in \{- \omega, i \omega, j \omega, k \omega \}$
 \item pairs $(a,b)$ where $ a,b \in \{\omega ^2 , i \omega ^2, j \omega ^2, k \omega ^2 \}$
 \item pairs $(a,b)$ where $ a,b \in \{-\omega ^2 , -i \omega ^2,- j \omega ^2,- k \omega ^2 \}$.
\end{itemize}

Counting the number of these pairs $(1 +1 +36 +16 +16 +16 +16)$, out of a group of order 192, we obtain a density of 17/32 (using the Chebotarev density theorem) for those primes whose image of Frobenius under the two different representations have equal traces.

We lift these representations to ${\rm Gal}(\overline{F}/F)$, and by Langlands~\cite{La80} we obtain the existence of tetrahedral cuspidal automorphic representations $\pi, \pi'$ such that the set 
\begin{align*}
\{v \mid v \text{ unramified for $\pi$ and $\pi'$, } a_v(\pi)=a_v(\pi')\} 
\end{align*}
has a density of $17/32$.

The element $(1,i)$ in ${\rm Gal}(K_1 K_2 / F)$ has components which have traces 2 and 0, respectively. This means that there exist places $v$ of $F$ where both 
\begin{align*}
a_v(\pi) &= 2 \\
a_v(\pi') &= 0
\end{align*}
hold. Therefore, $\pi$ and $\pi'$ cannot be twist-equivalent by a unitary Hecke character.

\end{document}